\documentclass[a4paper, 10pt]{article}
\usepackage{mathrsfs, amsmath, amsthm}
\usepackage{graphicx, xcolor}
\usepackage{tkz-graph, tkz-berge}
\usepackage{caption, subcaption}
\usepackage{array}
\usepackage{cases}
\makeatletter
\def\th@plain{%
  \upshape 
}
\makeatother

\makeatletter
\renewenvironment{proof}[1][\proofname]{\par
  \pushQED{\qed}%
  \normalfont \topsep6\p@\@plus6\p@\relax
  \trivlist
  \item[\hskip\labelsep
        \bfseries
    #1\@addpunct{.}]\ignorespaces
}{%
  \popQED\endtrivlist\@endpefalse
}
\makeatother

\newtheorem{theorem}{Theorem}
\numberwithin{theorem}{section}

\newtheorem{corollary}{Corollary}
\newtheorem{proposition}{Proposition}
\newtheorem{conjecture}{Conjecture}
\newtheorem*{conjecture*}{Conjecture}

\newtheorem{claim}{Claim}

\theoremstyle{definition}

\usepackage[left=25mm,top=25mm,bottom=25mm,right=25mm]{geometry}
\usepackage{newtxtext}
\usepackage[varvw]{newtxmath}

\usepackage{paralist, tablists}
\usepackage[inline]{enumitem}
\usepackage[square, numbers, sort&compress]{natbib}

\ifx\pdfoutput\undefined
 \usepackage[dvipdfm,%
  bookmarks=true,%
  bookmarksnumbered=true, 
  bookmarksopen=true, 
  plainpages=false,%
  pdfpagelabels,%
  colorlinks=true, 
  linkcolor=blue, 
  citecolor=blue,%
  hyperindex=true,
  urlcolor=black]{hyperref}
\else
 \usepackage[pdftex,%
  bookmarks=true,%
  bookmarksnumbered=true, 
  bookmarksopen=true, 
  plainpages=false,%
  pdfpagelabels,%
  colorlinks=true, 
  linkcolor=blue, 
  citecolor=blue,%
  anchorcolor=green,
  urlcolor= blue,
  breaklinks=true,
  hyperindex=true]{hyperref}
\fi

\newcounter{Hcase}
\newcounter{Hclaim}

\newcommand{\resetcounter}{\stepcounter{Hcase}\setcounter{case}{0}\stepcounter{Hclaim}\setcounter{claim}{0}}

\newcommand{\etal}{et~al.\ }

\def\int(#1){\mathrm{int}(#1)}
\def\ext(#1){\mathrm{ext}(#1)}
\def\Int(#1){\mathrm{Int}(#1)}
\def\Ext(#1){\mathrm{Ext}(#1)}
\def\ad(#1){\mathrm{ad}(#1)}
\def\mad(#1){\mathrm{mad}(#1)}
\def\la(#1){\mathrm{la}(#1)}

\begin{document}%
\title{List strong edge coloring of some classes of graphs}
\author{Watcharintorn Ruksasakchai\,\textsuperscript{a}\ \ \ Tao Wang\,\textsuperscript{b, }\footnote{{\tt Corresponding
author: wangtao@henu.edu.cn; iwangtao8@gmail.com}}\\
{\small \textsuperscript{a}Department of Mathematics Statistics and Computer Science}\\
{\small Faculty of Liberal Arts and Science, Kasetsart University}\\
{\small Kamphaeng Saen Campus, Nakhon Pathom, 73140, Thailand}\\
{\small \textsuperscript{b}Institute of Applied Mathematics}\\
{\small Henan University, Kaifeng, 475004, P. R. China}}
\date{March 29, 2017}
\maketitle

\begin{abstract}%
A {\em strong edge coloring} of a graph is a proper edge coloring in which every color class is an induced matching. The {\em strong chromatic index} of a graph is the minimum number of colors needed to obtain a strong edge coloring. In an analogous way, we can define the list version of strong edge coloring and list version of strong chromatic index. In this paper, we prove that if $G$ is a graph with maximum degree at most four and maximum average degree less than $3$, then the list strong chromatic index is at most $3\Delta + 1$, where $\Delta$ is the maximum degree of $G$. In addition, we prove that if $G$ is a planar graph with maximum degree at least $4$ and girth at least $7$, then the list strong chromatic index is at most $3\Delta$, where $\Delta$ is the maximum degree of $G$.

\end{abstract}
\section{Introduction}
A {\em proper edge coloring} of a graph is an assignment of colors to the edges such that adjacent edges receive distinct colors. The {\em chromatic index} $\chiup'(G)$ of a graph $G$ is the minimum number of colors needed to obtain a proper edge coloring of $G$. We denote the minimum and maximum degrees of vertices in $G$ by $\delta(G)$ and $\Delta(G)$, respectively. The well-known result on edge coloring is Vizing's theorem, which says that $\Delta(G) \leq \chiup'(G) \leq \Delta(G) + 1$.

An edge coloring is a {\em strong edge coloring} if every color class is an induced matching. That is, an edge coloring is {\em strong} if for each edge $uv$, the color of $uv$ is distinct from the colors of the edges (other than $uv$) incident with $N_{G}(u) \cup N_{G}(v)$, where $N_{G}(u)$ and $N_{G}(v)$ respectively denote the neighborhood of $u$ and $v$ in $G$. The {\em strong chromatic index} $\chiup_{s}'(G)$ of a graph $G$ is the minimum number of colors needed to obtain a strong edge coloring of $G$. The concept of strong edge coloring was introduced by Fouquet and Jolivet \cite{MR737086, MR776805}.

For a graph with maximum degree at most two, we can easily obtain the following result.
\begin{proposition}\label{Prop}%
If $G$ is a graph with maximum degree one, then $\chiup_{s}'(G) \leq 1$. If $G$ is a graph with maximum degree two, then $\chiup_{s}'(G) \leq 5$.
\end{proposition}

In 1985, Erd{\H{o}}s and Ne{\v s}et{\v r}il constructed graphs with strong chromatic index $\frac{5}{4}\Delta(G)^{2}$ when $\Delta(G)$ is even, and $\frac{1}{4}(5\Delta(G)^{2} - 2 \Delta(G) + 1)$ when $\Delta(G)$ is odd. Inspired by their construction, they proposed the following strong edge coloring conjecture during a seminar in Prague.
\begin{conjecture}[Erd{\H{o}}s and Ne{\v s}et{\v r}il \cite{MR975526}]\label{SECC}%
If $G$ is a graph with maximum degree $\Delta$, then
\[
\chiup_{s}'(G) \leq
\begin{cases}
\frac{5}{4} \Delta^{2}, & \text{if $\Delta$ is even;}\\
\frac{1}{4}(5\Delta^{2} - 2 \Delta + 1), & \text{if $\Delta$ is odd.}
\end{cases}
\]
\end{conjecture}

Andersen \cite{MR1189847} and Hor{\' a}k \etal \cite{MR1217390} independently confirmed the conjecture for $\Delta = 3$. Kostochka \etal \cite{MR3398865} proved that the strong chromatic index of a subcubic planar multigraph without loops is at most $9$. Some other classes of graphs have been investigated, such as degenerate graphs \cite{MR3056878, 2012arXiv1212.6092L, MR3279333, MR3386050, MR3208079} and Halin graphs \cite{MR2889505, MR2899886, MR2888094, MR2227736, MR2514906}. 

The {\em degree} of a vertex $v$ in $G$, denoted by $\deg(v)$, is the
number of incident edges of $v$ in $G$. A vertex of degree $k$, at most $k$ and at least $k$ is called a $k$-vertex, $k^{-}$-vertex and $k^{+}$-vertex, respectively. A {\em $k_{t}$-vertex} is a $k$-vertex adjacent to exactly $t$ vertices of degree two. Two distinct edges $e_{1} = uv$ and $e_{2}$ are {\em within distance two}, if $e_{2}$ is incident with at least one vertex in $N_{G}(u) \cup N_{G}(v)$. The {\em girth} of a graph $G$ is the length of a shortest cycle in $G$; if $G$ has no cycle we define the girth of $G$ to be $\infty$.

The {\em maximum average degree} $\mad(G)$ of a graph $G$ is the largest average degree of its subgraphs, that is,
\[
\mad(G) = \max_{H \subseteq G} \left\{\frac{2|E(H)|}{|V(H)|}\right\}.
\]

Hocquard \etal \cite{MR2823923, MR3101726} studied the strong chromatic index of subcubic graphs in terms of maximum average degree. The following result was proved in \cite{MR3101726} and it is the best result in the literature.
\begin{theorem}[Hocquard \etal \cite{MR3101726}]
  If $G$ is a subcubic graph, then
  \begin{enumerate}[label = (\alph*)]
    \item $\chiup_{s}'(G) \leq 6$ when $\mad(G) < \frac{7}{3}$;
    \item $\chiup_{s}'(G) \leq 7$ when $\mad(G) < \frac{5}{2}$;
    \item $\chiup_{s}'(G) \leq 8$ when $\mad(G) < \frac{8}{3}$;
    \item $\chiup_{s}'(G) \leq 9$ when $\mad(G) < \frac{20}{7}$.
  \end{enumerate}
\end{theorem}

In 1990, Faudree \etal \cite{MR1412876} studied the strong edge coloring of planar graphs. They gave an upper bound on the strong chromatic index in the following theorem.

\begin{theorem}[Faudree \etal \cite{MR1412876}]\label{4chiup}%
If $G$ is a planar graph, then $\chiup_{s}'(G) \leq 4\chiup'(G) \leq 4\Delta(G) + 4$.
\end{theorem}

The following result is a corollary of the proof of \autoref{4chiup}. 

\begin{theorem}\label{3chiup}%
If $G$ is a planar graph with girth at least $7$, then $\chiup_{s}'(G) \leq 3\chiup'(G) \leq 3\Delta(G) + 3$.
\end{theorem}

\autoref{4chiup} was proved by using the "very powerful tool" --- Four Color Theorem. \autoref{3chiup} was proved by using Gr{\"o}tzsch's theorem \cite{MR0116320}, which says that the chromatic number of a triangle-free planar graph is at most $3$. If someone wants to improve the upper bounds to $4\Delta(G)$ and $3\Delta(G)$ in \autoref{4chiup} and \autoref{3chiup} respectively by
adding some additional conditions, then all are attributed to the $\Delta(G)$-edge-coloring problem on planar graphs, and thus we do not
address them here. 

Hud\'{a}k \etal \cite{MR3176691} also proved the following two results.
\begin{theorem}[Hud\'{a}k \etal \cite{MR3176691}]\label{3Delta+5}%
If $G$ is a planar graph with girth at least $6$, then $\chiup_{s}'(G) \leq 3\Delta(G) + 5$.
\end{theorem}

\begin{theorem}[Hud\'{a}k \etal \cite{MR3176691}]\label{3Delta}%
If $G$ is a planar graph with girth at least $7$, then $\chiup_{s}'(G) \leq 3\Delta(G)$.
\end{theorem}

Bensmail \etal \cite{MR3281197} and the authors (in the first version of the current paper, see \url{http://arxiv.org/abs/1402.5677v1}) independently improved the upper bound in \autoref{3Delta+5} to $3\Delta(G) + 1$. 
\begin{theorem}[Bensmail \etal \cite{MR3281197}]
If $G$ is a planar graph with girth at least $6$, then $\chiup_{s}'(G) \leq 3\Delta(G) + 1$.
\end{theorem}

Similar to all the other kinds of coloring parameters, we can define the list strong edge coloring and list strong chromatic index $\chiup_{\mathrm{slist}}'$. In fact, most of the results for strong edge coloring are also true for list strong edge coloring, since the proofs just indicate the numbers of forbidden colors and available colors for uncolored edges in each step. 

Ma \etal \cite{MR3049780} investigated the list strong edge coloring of subcubic graphs, and gave the following results. 
\begin{theorem}[Ma \etal \cite{MR3049780}]
  If $G$ is a subcubic graph, then
  \begin{enumerate}[label = (\alph*)]
    \item $\chiup_{\mathrm{slist}}'(G) \leq 6$ when $\mad(G) < \frac{15}{7}$;
    \item $\chiup_{\mathrm{slist}}'(G) \leq 7$ when $\mad(G) < \frac{27}{11}$;
    \item $\chiup_{\mathrm{slist}}'(G) \leq 8$ when $\mad(G) < \frac{13}{5}$;
    \item $\chiup_{\mathrm{slist}}'(G) \leq 9$ when $\mad(G) < \frac{36}{13}$. \qed
  \end{enumerate}
\end{theorem}

Zhu \etal \cite{MR3233401} improved the above result to the followings. 

\begin{theorem}[Zhu \etal \cite{MR3233401}]
  If $G$ is a subcubic graph, then
  \begin{enumerate}[label = (\alph*)]
    \item $\chiup_{\mathrm{slist}}'(G) \leq 7$ when $\mad(G) < \frac{5}{2}$;
    \item $\chiup_{\mathrm{slist}}'(G) \leq 8$ when $\mad(G) < \frac{8}{3}$;
    \item $\chiup_{\mathrm{slist}}'(G) \leq 9$ when $\mad(G) < \frac{14}{5}$.\qed
  \end{enumerate}
\end{theorem}

In section~\ref{MadSec}, we prove that if $G$ is a graph with $\mad(G) < 3$ and $\Delta(G) \leq 4$, then $\chiup_{\mathrm{slist}}'(G) \leq 3\Delta(G) + 1$. In section~\ref{Girth7Sec}, we prove that if $G$ is a planar graph with maximum degree at least $4$ and girth at least $7$, then $\chiup_{\mathrm{slist}}'(G) \leq 3\Delta(G)$. 

\section{Graphs with maximum average degree less than $3$ and maximum degree at most $4$}\label{MadSec}
In this section, we prove the following result with restriction on maximum average degree and maximum degree. 
\begin{theorem}\label{mad<3}%
If $G$ is a graph with $\mad(G) < 3$ and $\Delta(G) \leq 4$, then $\chiup_{\mathrm{slist}}'(G) \leq 3\Delta(G) + 1$.
\end{theorem}

Note that Andersen \cite{MR1189847} and Hor{\' a}k \etal \cite{MR1217390} independently proved that the strong chromatic index of a subcubic graph is at most $10$. However, none of their proofs is true for the list strong chromatic index.  

\begin{corollary}\label{Delta4}%
If $G$ is a planar graph with girth at least $6$ and $\Delta(G) \leq 4$, then $\chiup_{\mathrm{slist}}'(G) \leq 3\Delta(G) + 1$.
\end{corollary}

Here, we do not prove \autoref{mad<3} directly, but alternately prove the corresponding result for original strong chromatic index, because the following proof can be trivially extended to the list strong chromatic index and it is easy for writing.    
\begin{theorem}\label{mad<3*}%
If $G$ is a graph with $\mad(G) < 3$ and $\Delta(G) \leq 4$, then $\chiup_{s}'(G) \leq 3\Delta(G) + 1$.
\end{theorem}
\begin{proof}%
Let $G$ be a minimum counterexample to the theorem. The minimum means that every proper subgraph  $H$ of $G$ admits a strong edge coloring with at most $3\Delta(H) + 1$ colors. By the minimality, the graph $G$ is connected and $\delta(G) \geq 1$. By \autoref{Prop}, the maximum degree is at least three.

\begin{claim}\label{delta}
The minimum degree of $G$ is at least two.
\end{claim}
\begin{proof}
Suppose that a $1$-vertex $v$ is adjacent to a vertex $u$. The graph $G - v$ has a strong edge coloring with at most $3\Delta(G) + 1$ colors. Note that the edge $uv$ has at most $3 \Delta(G)$ colored edges within distance two. Thus, we can assign an available color to $uv$, which is a contradiction.
\end{proof}

\begin{claim}\label{2-4}%
Every $2$-vertex is adjacent to a $4$-vertex.
\end{claim}
\begin{proof}
Suppose that a $2$-vertex $v$ is adjacent to two $3^{-}$-vertices $u$ and $w$. By the minimality of $G$, the graph $G - v$ has a strong edge coloring with at most $3\Delta(G) + 1$ colors. Both $uv$ and $wv$ have at most $3\Delta(G) - 1$ colored edges within distance two, thus both have at least two available colors, a contradiction.
\end{proof}

Since $\mad(G) < 3$, we have the following inequality:
\begin{equation*}
\sum_{v\, \in V(G)}(\deg(v) - 3) < 0, 
\end{equation*}
which implies $\Delta(G) = 4$ by \autoref{delta} and \autoref{2-4}.

\begin{claim}\label{Two4}%
Let $v_{1}$ be a $2$-vertex with $N_{G}(v_{1}) = \{v, w_{1}\}$. If $v$ is adjacent to at least two $2$-vertices, then $w_{1}$ is a $4$-vertex.
\end{claim}
\begin{proof}
Suppose to the contrary that $w_{1}$ is a $3^{-}$-vertex and $v$ is adjacent to another $2$-vertex $v_{2}$. By the minimality of $G$, the graph $G - v_{1}$ admits a strong edge coloring with at most $3\Delta(G) + 1$ colors. The edge $vv_{1}$ has at most $2\Delta(G) + 4$ colored edges within distance two, and the edge $v_{1}w_{1}$ has at most $2\Delta(G) + 3$ colored edges within distance two. Hence, the edge $vv_{1}$ has at least one available color and $v_{1}w_{1}$ has at least two available colors. Thus, we can color $vv_{1}$ and $v_{1}w_{1}$ in this order, which is a contradiction.
\end{proof}

\begin{claim}\label{No4-4}%
There is no $4_{4}$-vertex in $G$.
\end{claim}
\begin{proof}
Suppose that $v$ is a $4_{4}$-vertex with $N_{G}(v) = \{v_{1}, v_{2}, v_{3}, v_{4}\}$. By the minimality of $G$, the graph $G - v$ admits a strong edge coloring with at most $3\Delta(G) + 1$ colors. Note that every edge $vv_{i}$ has at most $\Delta(G) + 3$ colored edges within distance two, thus each edge has at least $2\Delta(G) - 2 = 6$ available colors, a contradiction.
\end{proof}

\begin{claim}\label{4-3}%
If a $2$-vertex is adjacent to a $4_{3}$-vertex, then the other neighbor is a $4_{1}$-vertex.  
\end{claim}
\begin{proof}%
Let $v$ be a $4_{3}$-vertex with three $2$-neighbors $v_{1}, v_{2}$ and $v_{3}$. Let $w_{1}$ be the other neighbor of $v_{1}$. By contradiction and \autoref{Two4} and \autoref{No4-4}, the vertex $w_{1}$ is a $4_{2}$- or $4_{3}$-vertex. By the minimality of $G$, the graph $G - v_{1}$ admits a strong edge coloring with at most $3\Delta(G) + 1$ colors. Now, we remove the color on edge $vv_{2}$ and denote the resulting coloring by $\sigma$. The edge $v_{1}w_{1}$ has at most $2\Delta(G) + 4$ colored edges within distance two, thus it has at least one available color, so we can assign a color to $v_{1}w_{1}$. After $v_{1}w_{1}$ was colored, both $vv_{1}$ and $vv_{2}$ have at least two available colors, so we can assign colors to them, which is a contradiction.
\end{proof}

Next, we use the discharging method to get a contradiction and complete the proof. We assign an initial charge $\deg(v) - 3$ to every vertex $v$, and then design appropriate discharging rules and redistribute charges among vertices, such that the final charge of every vertex is nonnegative, which derives a contradiction.

\paragraph{The Discharging Rules:}
\begin{enumerate}[label= (R\arabic*)]%
\item Every $4_{1}$-vertex sends $1$ to the adjacent $2$-vertex.
\item Every $4_{2}$-vertex sends $\frac{1}{2}$ to every adjacent $2$-vertex.
\end{enumerate}

\begin{itemize}%
\item {\bf Let $v$ be a $4$-vertex. }
If $v$ is a $4_{0}$-vertex, then the final charge is $4 - 3 = 1$. If $v$ is a $4_{1}$-vertex, then the final charge is $4 - 3 - 1 = 0$. If $v$ is a $4_{2}$-vertex, then the final charge is $4 - 3 - 2 \times \frac{1}{2} = 0$. If $v$ is a $4_{3}$-vertex, then the final charge is $4 - 3 = 1$. By \autoref{No4-4}, there is no $4_{4}$-vertex.

\item The final charge of every $3$-vertex is zero.
\item {\bf Let $v$ be a $2$-vertex. }
If $v$ is adjacent to a $4_{1}$-vertex, then the final charge is at least $2 - 3 + 1 = 0$.

By \autoref{4-3}, if $v$ is adjacent to a $4_{3}$-vertex, then it is adjacent to a $4_{1}$-vertex, but this case has been treated above. So we may assume that $v$ is not adjacent to any $4_{1}$- or $4_{3}$-vertex. By \autoref{2-4}, the $2$-vertex $v$ is adjacent to a $4$-vertex. By \autoref{Two4} and the excluded cases, the vertex $v$ is adjacent to two $4_{2}$-vertices, and then the final charge is $2 - 3 + 2 \times \frac{1}{2} = 0$. \resetcounter\qedhere
\end{itemize}
\end{proof}

\section{Planar graphs with girth at least $7$}\label{Girth7Sec}
The following is the main result in this section. 
\begin{theorem}\label{EP}%
Let $G$ be a planar graph with maximum degree {\em at most} $\Delta$, where $\Delta \geq 4$. If $G$ has girth at least $7$, then $\chiup_{\mathrm{slist}}'(G) \leq 3\Delta$.
\end{theorem}

Similar to the previous section, we do not prove the above result directly, but alternately prove the following corresponding result for original strong chromatic index. Note that the following result has been proved in \cite{MR3176691} by using Gr{\"o}tzsch's theorem. It is well known that Gr{\"o}tzsch's theorem is not true for the list vertex coloring, so the proof due to Hud{\' a}k \etal \cite{MR3176691} cannot be extended to the proof for \autoref{EP}. 

\begin{theorem}\label{EP*}%
Let $G$ be a plane graph with maximum degree {\em at most} $\Delta$, where $\Delta \geq 4$. If $G$ has girth at least $7$, then $\chiup_{s}'(G) \leq 3\Delta$.
\end{theorem}
\begin{proof}%
Let $G$ be a minimum counterexample to the theorem. The minimum means that every proper subgraph $H$ of $G$ has a strong edge coloring with at most $3\Delta$ colors. By the minimality, the graph $G$ is connected and $\delta(G) \geq 1$.

\begin{claim}\label{delta1}
Every $1$-vertex $v$ is adjacent to a $4^{+}$-vertex $u$; if $u$ is a $4$-vertex, then it is adjacent to three $\Delta$-vertices.
\end{claim}
\begin{proof}
Suppose that a $1$-vertex $v$ is adjacent to a vertex $u$. The graph $G - v$ admits a strong edge coloring with at most $3\Delta$ colors. If $uv$ has at most $3\Delta - 1$ colored edges within distance two, then we can extend the coloring to $G$, a contradiction. Hence, the edge $uv$ has at least $3\Delta$ colored edges within distance two, which implies the claim.
\end{proof}

\begin{claim}\label{2--4}%
Every $2$-vertex is adjacent to a $4^{+}$-vertex.
\end{claim}
\begin{proof}
Suppose that a $2$-vertex $v$ is adjacent to two $3^{-}$-vertices $u$ and $w$. By the minimality of $G$, the graph $G - v$ admits a strong edge coloring with at most $3\Delta$ colors. Both $uv$ and $wv$ have at most $2\Delta + 2 \leq 3\Delta - 2$ colored edges within distance two, thus each of them has at least two available colors, a contradiction.
\end{proof}

\begin{claim}\label{No33}%
Every vertex is adjacent to at least one $3^{+}$-vertex.
\end{claim}
\begin{proof}
Suppose that $v$ is a $\tau$-vertex with $N_{G}(v) = \{v_{1}, v_{2}, \dots, v_{\tau}\}$ and it is not adjacent to $3^{+}$-vertices. By the minimality of $G$, the graph $G - v$ admits a strong edge coloring with at most $3\Delta$ colors. Note that every edge $vv_{i}$ has at most $\Delta + \tau - 1 \leq 2\Delta - 1$ colored edges within distance two, thus each has at least $\Delta + 1$ available colors, a contradiction.
\end{proof}

\begin{claim}\label{412}%
If a $2$-vertex $v$ is adjacent to a $4$-vertex $u$ and a $2$-vertex $w$, then $u$ is a $4_{1}$-vertex.
\end{claim}
\begin{proof}
Suppose that $u$ is not a $4_{1}$-vertex. The graph $G - v$ admits a strong edge coloring with at most $3\Delta$ colors. The edge $uv$ has at most $(2\Delta + 2) + 1 = 2\Delta + 3$ colored edges within distance two, so we can assign a color to it. After the edge $uv$ was colored, the edge $wv$ has at most $\Delta + 4$ colored edges within distance two, so we also can assign an available color to it.
\end{proof}

\begin{claim}\label{No433}%
If a $2$-vertex $v$ is adjacent to a $4$-vertex $u$ and a $3$-vertex $w$, then $u$ is not a $4_{3}$-vertex.
\end{claim}
\begin{proof}
Suppose to the contrary that $u$ is a $4_{3}$-vertex. The graph $G - v$ admits a strong edge coloring with at most $3\Delta$ colors. The edge $wv$ has at most $2\Delta + 3$ colored edges within distance two, so we can assign a color to it. After $wv$ was colored, the edge $uv$ has at most $(\Delta + 4) + 3 = \Delta + 7 < 3\Delta$ colored edges within distance two, so we also can assign an available color to it.
\end{proof}

\begin{claim}\label{NoTwo42}%
Let $v_{1}$ be a $2$-vertex with $N_{G}(v_{1}) = \{v, w_{1}\}$. If $v$ is a $3_{2}$-vertex, then $w_{1}$ is a $5^{+}$-vertex or $4_{1}$-vertex.
\end{claim}
\begin{proof}
By contradiction and \autoref{No33}, suppose that $w_{1}$ is a $3^{-}$-vertex or $4_{2}$-vertex or $4_{3}$-vertex. By the minimality of $G$, the graph $G - v_{1}$ admits a strong edge coloring with at most $3\Delta$ colors. Let $v$ be adjacent to another $2$-vertex $v_{2}$. Now, we remove the color on edge $vv_{2}$ and denote the resulting coloring by $\sigma$. The edge $v_{1}w_{1}$ has at most $1 + (2\Delta + 2) = 2\Delta + 3 \leq 3\Delta - 1$ colored edges within distance two, thus we can assign an available color to $v_{1}w_{1}$. After the edge $v_{1}w_{1}$ was colored, the edge $vv_{1}$ has at most $(\Delta + 1) + 4 = \Delta + 5$ colored edges within distance two, and $vv_{2}$ has at most $(\Delta + 1) + \Delta = 2\Delta + 1$ colored edges within distance two. Thus, both $vv_{1}$ and $vv_{2}$ have at least three available colors, and we can extend $\sigma$ to a strong edge coloring of $G$, which is a contradiction.
\end{proof}

\begin{claim}\label{U3}
Suppose that $v$ is adjacent to exactly one $3^{+}$-vertex. If $u$ is a $2^{-}$-neighbor of $v$, then it is a $2$-vertex, and the other neighbor of $u$ is a $4^{+}$-vertex. 
\end{claim}
\begin{proof}
If $u$ is a $1$-vertex, then $uv$ has at most $\Delta + 2(\Delta - 2) = 3\Delta - 4$ edges (other than $uv$) within distance two, so we can first color $G - u$ and extend the coloring to $G$, a contradiction. So we may assume that $v$ has no $1$-neighbor. 

Suppose that $w$ is the other neighbor of $u$ and it is a $3^{-}$-vertex. The graph $G - u$ admits a strong edge coloring with at most $3\Delta$ colors. The edge $uw$ has at most $2\Delta + (\Delta - 1) = 3\Delta - 1$ colored edges within distance two, so we can assign an available color to $uw$. After the edge $uw$ was colored, the edge $uv$ has at most $\Delta + 2(\Delta - 2) + 3 = 3\Delta - 1$ colored edges within distance two, so we also can assign an available color to $uv$, and then we obtain a strong edge coloring of $G$, a contradiction. 
\end{proof}

\begin{claim}\label{U33}
Let $k \geq 5$ be an integer. Suppose that a $k$-vertex $v$ is adjacent to exactly two $3^{+}$-vertices. Then $v$ is adjacent to at most $k - 5$ pendent vertices. Moreover, if $v$ is adjacent to exactly $k - 5$ pendent vertices, then one of the adjacent $2$-vertex is adjacent to a $3_{1}$-vertex or a $4^{+}$-vertex (other than $v$). 
\end{claim}
\begin{proof}
Let $\ell$ be the number of adjacent $1$-vertices. Suppose to the contrary that $\ell \geq k - 4 \geq 1$, and $uv$ is a pendent edge incident with $v$. Note that $uv$ has at most $2\Delta + (\ell - 1) + 2(k - 2 - \ell) = 2\Delta + 2k - \ell - 5 \leq 3\Delta$ edges (other than $uv$) within distance two, a contradiction. 

Assume that $v$ is adjacent to exactly $k - 5$ pendent vertices and three $2$-vertices $v_{1}, v_{2}, v_{3}$. Let $N_{G}(v_{1}) = \{v, w_{1}\}$, $N_{G}(v_{2}) = \{v, w_{2}\}$ and $N_{G}(v_{3}) = \{v, w_{3}\}$. Suppose that all the vertices $w_{1}, w_{2}$ and $w_{3}$ are $2$- or $3_{2}$-vertices. By the minimality, $G - v_{1}$ has a strong edge coloring with at most $3\Delta$ colors. Now, we remove the colors on $v_{2}w_{2}$ and $v_{3}w_{3}$. The edge $vv_{1}$ has at most $2\Delta + (k-3) + 2 \leq 3\Delta - 1$ colored edges within distance two, so we can assign an available color to $vv_{1}$. After $vv_{1}$ was colored, we can assign available colors to $v_{1}w_{1}, v_{2}w_{2}, v_{3}w_{3}$ in this order.
\end{proof}

Euler's formula can be rewritten to the following equality:
\begin{equation}%
\sum_{v\, \in V(G)}\left(\frac{5}{2}\deg(v) - 7\right) + \sum_{f\, \in\, F(G)}(\deg(f) - 7) = - 14.
\end{equation}

Similar to the previous section, we use the discharging method to get a contradiction. We first assign an initial charge $\frac{5}{2}\deg(v) - 7$ to every vertex $v$ and $\deg(f) - 7$ to every face $f$. We design appropriate discharging rules and redistribute charges among vertices and faces, such that the final charge of every vertex and every face is nonnegative, which leads to a contradiction.

\paragraph{The Discharging Rules:}
\begin{enumerate}[label= (R\arabic*)]%
\item\label{1face} Every $1$-vertex receives $2$ from its incident face.
\item\label{1vertex} Every $1$-vertex receives $\frac{5}{2}$ from the adjacent vertex.
\item\label{41} Every $4_{1}$-vertex sends $3$ to the adjacent $2$-vertex.
\item\label{42} Every $4_{2}$-vertex sends $\frac{3}{2}$ to every adjacent $2$-vertex.
\item\label{43} Every $4_{3}$-vertex sends $1$ to every adjacent $2$-vertex.
\item\label{5+2} If a $2$-vertex $v$ is adjacent to a $5^{+}$-vertex $u$ and a $2$-vertex, then $v$ receives $2$ from $u$.
\item\label{423} If a $2$-vertex $v$ is adjacent to a $4_{2}$-vertex $u$ and a $3_{1}$-vertex $w$, then $v$ receives $\frac{1}{2}$ from $w$.
\item\label{5+31} If a $2$-vertex $v$ is adjacent to a $5^{+}$-vertex $u$ and a $3_{1}$-vertex $w$, then $v$ receives $\frac{3}{2}$ from $u$ and $\frac{1}{2}$ from $w$.
\item\label{5+32} If a $2$-vertex $v$ is adjacent to a $5^{+}$-vertex $u$ and a $3_{2}$-vertex $w$, then $v$ receives $2$ from $u$.
\item\label{5+4+} If a $2$-vertex $v$ is adjacent to a $5^{+}$-vertex $u$ and a $4^{+}$-vertex $w$, then $v$ receives $1$ from $u$.
\end{enumerate}

\begin{itemize}%
\item If $f$ is a face incident with $t$ vertices of degree one, then the degree of $f$ is at least $7 + 2t$, and then the final charge of $f$ is at least $(7 + 2t) - 7 - 2t = 0$ by \ref{1face}.

\item If $v$ is a $1$-vertex, then it receives $2$ from its incident face and $\frac{5}{2}$ from the adjacent vertex, and then the final charge is at least $\frac{5}{2} - 7 + 2 + \frac{5}{2} = 0$ by \ref{1face} and \ref{1vertex}.

\item {\bf Let $v$ be a $2$-vertex.} By \autoref{delta1}, the vertex $v$ is not adjacent to $1$-vertices. By \autoref{2--4}, the vertex $v$ is adjacent to a $4^{+}$-vertex.

If $v$ is adjacent to a $4_{1}$-vertex, then the final charge is at least $\frac{5}{2} \times 2 - 7 + 3 = 1$ by \ref{41}. So we may assume that $v$ is not adjacent to $4_{1}$-vertices.

If $v$ is adjacent to a $4^{+}$-vertex $u$ and a $2$-vertex $w$, then according to \autoref{412} and excluded case, the vertex $u$ is a $5^{+}$-vertex, and then the final charge of $v$ is $\frac{5}{2} \times 2 - 7 + 2 = 0$ by \ref{5+2}.

If $v$ is adjacent to a $4$-vertex $u$ and a $3$-vertex $w$, then according to \autoref{No433}, \ref{NoTwo42} and excluded cases, $u$ is a $4_{2}$-vertex and $w$ is a $3_{1}$-vertex, and then the final charge is $\frac{5}{2} \times 2 - 7 + \frac{3}{2} + \frac{1}{2} = 0$ by \ref{42} and \ref{423}.

If $v$ is adjacent to a $5^{+}$-vertex $u$ and a $3_{1}$-vertex $w$, then the final charge is $\frac{5}{2} \times 2 - 7 + \frac{3}{2} + \frac{1}{2} = 0$ by \ref{5+31}.

If $v$ is adjacent to a $5^{+}$-vertex $u$ and a $3_{2}$-vertex $w$, then the final charge is $\frac{5}{2} \times 2 - 7 + 2 = 0$ by \ref{5+32}.

Now, we assume that $v$ is adjacent to two $4^{+}$-vertices. If $v$ is adjacent to two $4_{2}$-vertices, then the final charge is $\frac{5}{2} \times 2 - 7 + \frac{3}{2} + \frac{3}{2} = 1$ by \ref{42}. If $v$ is adjacent to a $4_{2}$-vertex and a $4_{3}$-vertex, then the final charge is $\frac{5}{2} \times 2 - 7 + \frac{3}{2} + 1 = \frac{1}{2}$ by \ref{42} and \ref{43}. If $v$ is adjacent to two $4_{3}$-vertices, then the final charge is $\frac{5}{2} \times 2 - 7 + 1 + 1 = 0$ by \ref{43}. If $v$ is adjacent to a $5^{+}$-vertex, then the final charge is at least $\frac{5}{2} \times 2 - 7 + 1 + 1 = 0$ by \ref{5+4+}.

\item {\bf Let $v$ be a $3$-vertex.} By \autoref{delta1}, the vertex $v$ is not adjacent to $1$-vertices. If $v$ is a $3_{0}$-vertex, then the final charge is $\frac{5}{2} \times 3 - 7 = \frac{1}{2}$. If $v$ is a $3_{1}$-vertex, then the final charge is at least $\frac{5}{2} \times 3 - 7 - \frac{1}{2} = 0$ by \ref{423} and \ref{5+31}. If $v$ is a $3_{2}$-vertex, then the final charge is $\frac{5}{2} \times 3 - 7 = \frac{1}{2}$ by \ref{5+32}.

\item {\bf Let $v$ be a $4$-vertex.} If $v$ is adjacent to a $1$-vertex, then according to \autoref{delta1} and \ref{1vertex}, the final charge is $\frac{5}{2} \times 4 - 7 - \frac{5}{2} = \frac{1}{2}$. So we may assume that $v$ is not adjacent to $1$-vertices. If $v$ is a $4_{1}$-vertex, then the final charge is $\frac{5}{2} \times 4 - 7 - 3 = 0$ by \ref{41}. If $v$ is a $4_{2}$-vertex, then the final charge is $\frac{5}{2} \times 4 - 7 - 2 \times \frac{3}{2} = 0$ by \ref{42}. If $v$ is a $4_{3}$-vertex, then the final charge is $\frac{5}{2} \times 4 - 7 - 3 \times 1 = 0$ by \ref{43}. No $4$-vertex is adjacent to four $2$-vertices. If $v$ is adjacent to four $3^{+}$-vertices, then the final charge is $\frac{5}{2} \times 4 - 7 = 3$.

\item {\bf Let $v$ be a $k$-vertex with $k \geq 5$.} If $v$ is adjacent to at least three $3^{+}$-vertices, then the final charge is at least $\frac{5}{2}k - 7 - \frac{5}{2}(k-3) = \frac{1}{2}$. By \autoref{No33}, we may assume that $v$ is adjacent to one or two $3^{+}$-vertices. If $v$ is adjacent to exactly one $3^{+}$-vertex, then $v$ has the final charge at least $\frac{5}{2}k - 7 - (k - 1) = \frac{3}{2}k - 6 > 0$ by \autoref{U3} and \ref{5+4+}.

Assume that $v$ is adjacent to exactly two $3^{+}$-vertices. Let $\ell$ be the number of adjacent $1$-vertices. If $0 \leq \ell \leq k - 6$, then the final charge of $v$ is at least $\frac{5}{2}k - 7 - \frac{5}{2}\ell - 2(k - 2 - \ell) = \frac{1}{2}(k - \ell) - 3 \geq 0$. By \autoref{U33}, we have that $\ell = k - 5$, and $v$ is adjacent to exactly three $2$-vertices, one of which is adjacent to a $3_{1}$-vertex or a $4^{+}$-vertex (other than $v$). Thus, the final charge is at least $\frac{5}{2}k - 7 - \frac{5}{2}(k - 5) - \frac{3}{2} - 2 - 2 = 0$ by \ref{1vertex}, \ref{5+31}, \ref{5+32} and \ref{5+4+}. \qedhere
\end{itemize}
\end{proof}

\begin{corollary}%
If $G$ is a planar graph with maximum degree at least $4$ and girth at least $7$, then $\chiup_{s}'(G) \leq 3\Delta(G)$.
\end{corollary}

\section{Concluding remark}
We could not find any example regarding the tightness of the upper bounds we have shown. So we expect that the upper bounds might be further improved.  

\vskip 3mm \vspace{0.3cm} \noindent{\bf Acknowledgments.} This project was supported by the National Natural Science Foundation of China (11101125) and partially supported by the Fundamental Research Funds for Universities in Henan.

\end{document}